\documentclass{amsart}
\usepackage{a4wide,amssymb}

\let\ge\geqslant

\def\1{^{-1}}

\newtheorem{theorem}{Theorem}[section]

\newtheorem{proposition}[theorem]{Proposition}

\theoremstyle{definition}

\newtheorem{corollary}[theorem]{Corollary}

\newtheorem{conjecture}[theorem]{Conjecture}
\theoremstyle{remark}

\numberwithin{equation}{section}

\begin{document}

\title[]{Some explicit expressions concerning formal group laws}
\begin{abstract}
 This paper provides some explicit expressions concerning the formal group laws of the following cohomology theories: $BP$, The Brown-Peterson cohomology, $G(s)$, the cohomology theory obtained from $BP$ by putting $v_i=0$ for all $i\geq 1$ with $i \neq s$, the Morava $K$-theory and the  Abel cohomology.  The coefficient ring of $p$-typization of the universal Abel formal group law is computed in low dimensions.
 One consequence of the latter is that even though the Abel formal group laws come from Baas-Sullivan theory, the corresponding p-typical version does not.
\end{abstract}
\author{Malkhaz Bakuradze \& Mamuka Jibladze }
\address{Faculty of exact and natural sciences, Iv. Javakhishvili Tbilisi State University, Georgia }
\email{malkhaz.bakuradze@tsu.ge \& mamuka.jibladze@rmi.ge}
\thanks{First author was supported by Rustaveli NSF, DI/16/5-103/12 and by Volkswagen Foundation, Ref.: I/84 328.}

\maketitle

\section{Introduction}

Let us start with some necessary definitions. The main reference for formal group laws is \cite{BUCH1}, see also \cite{RAV}.

A formal group law over a commutative ring with unit $R$ is a power series $F(x,y)\in R[[x,y]]$ satisfying

\medskip

(i) $F(x,0)=F(0,x)=x$,

(ii) $F(x,y)=F(y,x)$,

(iii) $F(x,F(y,z))=F(F(x,y),z)$.

\medskip

Let $F$ and $G$ be formal group laws. A homomorphism from $F$ to $G$ is a power series $\nu(x)\in R[[x]]$ with constant term $0$ such that
$$\nu(F(x,y))=G(\nu(x),\nu(y)).$$

It is an isomorphism if $\nu'(0)$ (the coefficient at $x$) is a unit in $R$, and a strict isomorphism if the coefficient at $x$ is 1.

If $F$ is a formal group law over a commutative $\mathbb{Q}$-algebra $R$, then it is strictly isomorphic to the additive formal group law $x+y$. In other words, there is a strict isomorphism $l(x)$ from $F$ to the additive formal group law, called the logarithm of $F$, so that $F(x,y)=l^{-1}(l(x)+l(y))$.  The inverse to logarithm is called the exponential of $F$.

The logarithm $l(x)\in R \otimes \mathbb{Q}[[x]]$ of a formal group law $F$ is given by
$$
l(x)=\int_{0}^{x}\frac{dt}{\omega(t)},\,\,\,\,\,\omega(x)=\frac{\partial F(x,y)}{\partial y}(x,0).
$$

There is a ring $L$, called the universal Lazard ring, and a universal formal group law $F(x,y)=\sum a_{ij}x^iy^j$ defined over $L$. This means that for any formal group law $G$ over any commutative ring with unit $R$ there is a unique ring homomorphism $r:L\rightarrow R$ such that $G(x,y)=\sum r(a_{ij})x^iy^j$.

The formal group law of geometric cobordism was introduced in \cite{NOV}. Following Quillen we will identify it with the universal Lazard formal group law as it is proved in \cite{Q} that the coefficient ring of complex cobordism $MU_*=\mathbb{Z}[x_1,x_2,...]$, $|x_i|=2i$ is naturally isomorphic as a graded ring to the universal Lazard ring.

\bigskip

For a power series of the form $m(x)=x+m_1x^2+m_2x^3+...$, its composition inverse $e(x)=x+e_1x^2+e_2x^3+...$ is given by
$$
e_n=\sum_{\substack{
k_1,k_2,...\ge0\\
k_1+2k_2+3k_3+...=n
}}(-1)^{k_1+k_2+...}\frac{(n+k_1+k_2+...)!}{(n+1)!k_1!k_2!\cdots}m_1^{k_1}m_2^{k_2}\cdots
$$
Coefficients of the corresponding formal group $F(x,y)=e(m(x)+m(y))=x+y+\alpha_{11}xy+...$ are thus given by
\begin{multline}\label{general}
\alpha_{ij}=\sum_{\nu_1+2\nu_2+...=i+j-1}\\
\sum_{\substack{
i_0+2i_1+3i_2+...=i,\\
j_0+2j_1+3j_2+...=j,\\
i_1+j_1+k_1=\nu_1,i_2+j_2+k_2=\nu_2,...
}}(-1)^{k_1+k_2+...}\frac{(i_0+j_0+k_1+i_1+j_1+k_2+...-1)!}{i_0!j_0!k_1!i_1!j_1!k_2!\cdots}\\
m_1^{\nu_1}m_2^{\nu_2}\cdots
\end{multline}

The rest of paper is organized as follows. In Section 2 we give some explicit expressions concerning universal $p$-typical formal group law $F_{BP}$. In Section 3 we consider the Morava case $F_{G(s)}$. Section 4 is devoted to Abel universal formal group law $\mathcal{F}_{Ab}$. Finally in Section 5 we consider $p$-typization of $\mathcal{F}_{Ab}$ and compute its coefficient ring in low dimensions.  Our computation imply that even though $\mathcal{F}_{Ab}$ come from Baas-Sullivan theory, the corresponding p-typical version does not.
The main reference for the last result is the original article by Baas \cite{Baas} and for bordism
theory is the book by Stong \cite{St}. Sullivan defined a similar notion of singular manifolds in his
work on the Hauptvermutung \cite{Sul}.

\section{Formal group law in Brown-Peterson cohomology}

Denote by $F_{BP}(x,y)=\sum\alpha_{ij}x^iy^j$ the formal group law of Brown-Peterson cohomology $BP$ \cite{B-P} and let
$\log_{BP}(x)=x+l_1x^p+l_2x^{p^2}+l_3x^{p^3}+\cdots$ be the logarithm of $F_{BP}.$ The first choice of the generators of $BP_*=\mathbb{Z}_{(p)}[v_1,v_2\cdots]$, $|v_i|=2(p^n-1)$ was given by Hazewinkel \cite{HAZ}. The coefficients $l_i\in BP_*\otimes\mathbb{Q}$
are related to the Hazewinkel generators $v_i\in BP_*\hookrightarrow BP_*\otimes\mathbb{Q}$ through the recursive equation \cite{HAZ,RAV}

\begin{equation}
\label{eq:recursive}
 pl_n=v_n+v_{n-1}^{p}l_1+v_{n-2}^{p^2}l_2+\cdots +v_1^{p^{n-1}}l_{n-1}.
\end{equation}

\medskip

The following is easily checked by explicit computation:

 \begin{proposition}
 \label{solution}
 An explicit solution of \eqref{eq:recursive}, i. e. an expression of $l_n$ through the $v_k$ is given by
 $$l_n=\sum_{k=1}^{n}\,\,\,\sum_{\begin{subarray}{1}
 n_1, \cdots, n_k>0 \\n_1+ \cdots +n_k=n \end{subarray}}
 v_{n_1}v_{n_2}^{p^{n_1}}v_{n_3}^{p^{n_1+n_2}}\cdots v_{n_k}^{p^{n_1+n_2+\cdots n_{k-1}}}/p^k.$$
 \end{proposition}

\medskip

 Thus there are on the whole $2^{n-1}$ summands for $l_n$. For example,

 $$l_4=v_4/p+v_1v_3^p/p^2+v_2v_2^{p^2}/p^2+v_3v_1^{p^3}/p^2+v_1v_1^pv_2^{p^2}/p^3+v_1v_2^{p}v_1^{p^3}/p^3+v_2v_1^{p^2}v_1^{p^3}/p^3+
 v_1v_1^pv_1^{p^2}v_1^{p^3}/p^4.$$

\medskip

\begin{proposition}
 \label{main}
 The coefficient $\alpha_{ij}$ of the formal group law $F_{BP}$ at $x^iy^j$ is given by
\begin{multline*}
 \alpha_{ij}=\sum_{(p-1)\left(\nu_1+\frac{p^2-1}{p-1}\nu_2+\frac{p^3-1}{p-1}\nu_3+...\right)=i+j-1}\\
\sum_{\substack{
i_0+pi_1+p^2i_2+\cdots =i, \\
j_0+pj_1+p^2j_2+\cdots =j,\\
i_1+j_1+k_1=\nu_1,i_2+j_2+k_2=\nu_2,...
}}
 (-1)^{k_1+k_2+...}\frac{(i_0+j_0+i_1+j_1+k_1+i_2+j_2+k_2+...-1)!}{i_0!j_0!i_1!j_1!k_1!i_2!j_2!k_2!\cdots}\\
 l_1^{\nu_1}l_2^{\nu_2}l_3^{\nu_3}\cdots;
\end{multline*}
in particular, $\alpha_{ij}$ is nonzero only when $i+j-1$ is a multiple of $p-1$.
\end{proposition}

\begin{proof}
Taking in \eqref{general} $l(x)=\log_{BP}(x)$ with $m_{p^k-1}=l_k$ and all other $m_i$ equal to zero, we obtain the desired expression.
\end{proof}

It follows that for any $0<k<p$ the coefficient of $\alpha_{kp^n\,(p-k)p^n}$ at $l_{n+1}$ is equal to $-\binom{p^{n+1}}{kp^{n}}$. Moreover we know from \eqref{eq:recursive} that $l_{n+1}=\frac{1}{p}v_{n+1}+\text{decomposables.}$ Hence we have
$$
v_{n+1}=-\frac{p}{\binom{p^{n+1}}{kp^{n}}}\alpha_{kp^n\,(p-k)p^n}+\text{decomposables}.
$$

The binomial coefficient $\binom{p^{n+1}}{k_np^{n}}$ is divisible by $p$ and not divisible by $p^2$: Recall that for any $m$ the $p$-adic valuation of $m!$ (i. e. the largest power of $p$ dividing $m!$) is given by
$$
\biggl[\frac{m}{p}\biggr]+\biggl[\frac{m}{p^2}\biggr]+\biggl[\frac{m}{p^3}\biggr]+\cdots,
$$
where $[x]$ denotes the integer part of $x$. Using this respectively for $m=p^{n+1}$, $m=kp^n$ and $m=(p-k)p^n$ gives the $p$-adic valuation of
$\binom{p^{n+1}}{kp^{n}}=\frac{p^{n+1}!}{(kp^n)!((p-k)p^n)!}$ as follows

\begin{align*}
&\biggl[\frac{p^{n+1}}{p}\biggr]&+\biggl[\frac{p^{n+1}}{p^2}\biggr]&+\biggl[\frac{p^{n+1}}{p^3}\biggr]+&\cdots+\biggl[\frac{p^{n+1}}{p^n}\biggr]&+
\biggl[\frac{p^{n+1}}{p^{n+1}}\biggr]\\
-&\biggl[\frac{kp^n}{p}\biggr]    &-\biggl[\frac{kp^n}{p^2}\biggr]    &-\biggl[\frac{kp^n}{p^3}\biggr]-&\cdots-\biggl[\frac{kp^n}{p^n}\biggr]&\\
-&\biggl[\frac{(p-k)p^n}{p}\biggr]&-\biggl[\frac{(p-k)p^n}{p^2}\biggr]&-\biggl[\frac{(p-k)p^n}{p^3}\biggr]-&\cdots-\biggl[\frac{(p-k)p^n}{p^n}\biggr]&
=1
\end{align*}

Therefore the coefficient  $\frac{p}{\binom{p^{n+1}}{kp^{n}}}$ belongs to $\mathbb{Z}_{(p)}.$

As an application one can construct particular polynomial generators for $BP_*$ as follows.

\begin{corollary} One can construct polynomial generators for $BP_*$ from coefficients $\alpha_{ij}$
of the universal $p$-typical formal group law
 $$BP_*\cong \mathbb{Z}_{(p)}[\alpha_{k_0 \,p-k_0},\alpha_{k_1p \,(p-k_1)p},\alpha_{k_2p^2 \,(p-k_2)p^2},\cdots]$$
for any $0<k_n<p$, $n=0,1,2,\cdots.$
\end{corollary}

\medskip

Here are some examples of expression of the Hazewinkel generators through the elements $\alpha_{kp^n\,(p-k)p^n}$:

\begin{align*}
\intertext{- for $p=2$,}
v_1=&-\alpha_{11},\\
v_2=&-\frac{1}{3}\alpha_{22}+\frac{4}{3}\alpha^3_{11},\\
v_3=&-\frac{1}{35}\alpha_{44}+\frac{302}{315}\alpha^4_{11}\alpha_{22}-\frac{170}{63}\alpha^7_{11},\\
\intertext{- for $p=3$,}
v_1=&-\alpha_{12},\\
v_2=&-\frac{1}{28}\alpha_{36}+\frac{27}{28}\alpha^4_{12},\\
v_3=&-\frac{1}{1562275}\alpha_{9\,18}+\frac{90115407}{17147530400}\alpha_{12}\alpha^3_{36}+\frac{27811961973}{17147530400}\alpha^5_{12}\alpha^2_{36}\\
&-\frac{328516118111}{3429506080}\alpha^9_{12}\alpha_{36}-\frac{20612623337247}{17147530400}\alpha^{13}_{12},\\
\end{align*}
etc.

\bigskip

\section{The Morava case}

To apply the above formul{\ae} to calculation of the Morava $K$-theories, consider the
theories $G(s)$ \cite{RUD} obtained from $BP$ by putting $v_i=0$ for all $i\geq 1$ with $i \neq s$.
Since for $G(s)$ theory $v_s$ plays the r\^{o}le of just a bookkeeping variable, let us drop it, i. e. put $v_s=1$.
Thus the logarithm of $G(s)$ can be just written as
$$
x+\frac{x^{p^s}}{p}+\frac{x^{p^{2s}}}{p^2}+\frac{x^{p^{3s}}}{p^3}+\cdots.
$$
Thus taking in Proposition \ref{main} $l_{ks}=1/p^k$ and $l_i=0$ otherwise, we readily obtain a formula for the coefficients.

\begin{proposition}
 \label{forG}
 The coefficient $\alpha_{ij}$ of the formal group law $F_{G(s)}$ at $x^iy^j$ is given by
 \begin{multline*}
 \alpha_{ij}=\sum_{(p^s-1)\left(\nu_1+\frac{p^{2s}-1}{p^s-1}\nu_2+\frac{p^{3s}-1}{p^s-1}\nu_3+...\right)=i+j-1}\\
\sum_{\substack{
i_0+p^si_1+p^{2s}i_2+\cdots =i, \\
j_0+p^sj_1+p^{2s}j_2+\cdots =j,\\
i_1+j_1+k_1=\nu_1,i_2+j_2+k_2=\nu_2,...
}}
 (-1)^{k_1+k_2+...}\frac{(i_0+j_0+i_1+j_1+k_1+i_2+j_2+k_2+...-1)!}{i_0!j_0!i_1!j_1!k_1!i_2!j_2!k_2!\cdots p^{\nu_1+2\nu_2+...}}
\end{multline*}
\end{proposition}
\qed

\

Another convenient form of this formula is
$$
\alpha_{i,k(p^s-1)+1-i}=\kern-7em
\sum_{\substack{
i_0+p^si_1+p^{2s}i_2+\cdots =i, \\
j_0+p^sj_1+p^{2s}j_2+\cdots =k(p^s-1)+1-i,\\
k_1+\frac{p^{2s}-1}{p^s-1}k_2+\frac{p^{3s}-1}{p^s-1}k_3+...=k-i_1-j_1-\frac{p^{2s}-1}{p^s-1}(i_2+j_2)-\frac{p^{3s}-1}{p^s-1}(i_3+j_3)-...
}}
\kern-11em(-1)^{k_1+k_2+...}\frac{(i_0+j_0+i_1+j_1+k_1+i_2+j_2+k_2+...-1)!}{i_0!j_0!i_1!j_1!k_1!i_2!j_2!k_2!\cdots p^{i_1+j_1+k_1+2(i_2+j_2+k_2)+...}},
$$
all other $\alpha_{ij}$ being zero.

Morava K-theory formal group is then reduction of the above modulo $p$.

Another way of computing the latter formal group is via the Ravenel recursive relation \cite{RAV} involving Witt symmetric functions,
$$
F(x_1,x_2,...)=F(W^{(1)}(x_i),W^{(p)}(x_i)^{p^{s-1}},W^{(p^2)}(x_i)^{p^{2(s-1)}},...)
$$
where $W^{(n)}$ are the symmetric functions defined by
$$
\frac{\sum_ix_i^n}n=\sum_{d|n}\frac{W^{(\frac nd)}(x_i)^d}d
$$
for all $n$. For example, with two variables one has
\begin{align*}
W^{(1)}=&x+y,\\
W^{(2)}=&-xy,\\
W^{(3)}=&-(x^2y+xy^2),\\
W^{(4)}=&-(x^3y+2x^2y^2+xy^3),\\
W^{(5)}=&-(x^4y+2x^3y^2+2x^2y^3+xy^4),\\
W^{(6)}=&-(x^5y+3x^4y^2+4x^3y^3+3x^2y^4+xy^5),\\
...
\end{align*}

Some examples of how this recursion actually works: for given $p$ and $s$ let us denote $w_k=W^{(p^k)}(x,y)^{p^{k(s-1)}}$; then

\noindent For $p=2$, $s=1$
\begin{align*}
F(x,y)=&w_0(=x+y)+w_1(=xy)+w_0w_1(=x^2y+xy^2)+(w_2+w_0^2w_1)(=0)\\
&+(w_0w_2+w_0w_1^2)(=x^4y+xy^4)+(w_1w_2+w_0^2w_2+w_0^4w_1)(=x^4y^2+x^2y^4)\\
&+w_0w_1w_2(=x^5y^2+x^4y^3+x^3y^4+x^2y^5)+...;\\
\intertext{For $p=2$, $s=2$}
F(x,y)=&w_0(=x+y)+w_1(=x^2y^2)+w_0^2w_1^2(=x^6y^4+x^4y^6)+w_2(=x^{12}y^4+x^4y^{12})\\
&+w_0^6w_1^4(=x^{14}y^8+x^{12}y^{10}+x^{10}y^{12}+x^8y^{14})+(w_0^4w_1^6+w_0^{12}w_1^4)(=x^{20}y^8+x^8y^{20})+...;\\
\intertext{For $p=2$, $s=3$}
F(x,y)=&w_0(=x+y)+w_1(=x^4y^4)+w_0^4w_1^4(=x^{20}y^{16}+x^{16}y^{20})+w_2(=x^{48}y^{16}+x^{16}y^{48})\\
&+w_0^{48}w_1^{16}(=x^{112}y^{64}+x^{96}y^{80}+x^{80}y^{96}+x^{64}y^{112})+...;\\
\intertext{For $p=3$, $s=1$}
F(x,y)=&w_0(=x+y)+w_1(=-x^2y-xy^2)+(-w_0^2w_1)(=-x^4y-xy^4)\\
&+(-w_0w_1^2+w_0^4w_1)(=-x^6y-x^4y^3-x^3y^4-xy^6)+...\\
\intertext{For $p=3$, $s=2$}
F(x,y)=&w_0(=x+y)+w_1(=-x^6y^3-x^3y^6)\\
&+(-w_0^6w_1^3)(=x^{24}y^9-x^{21}y^{12}+x^{18}y^{15}+x^{15}y^{18}-x^{12}y^{21}+x^9y^{24})+...
\end{align*}

\bigskip

Let us give the following two approximations to the formal group law of Morava $K$-theory.

\begin{proposition}(see \cite{BP1})
\label{moravaFGL}
For the formal group law in mod $p$ Morava K-theory $K^*(s)$ at prime $p$ and $s>1$, we have

$$
F(x,y) \equiv x + y - v_s \sum_{0<j<p}p^{-1}\binom{p}{j}(x^{p^{s-1}})^{j}(y^{p^{s-1}})^{p-j}
$$
modulo $x^{p^{2(s-1)}}$ (or modulo $y^{p^{2(s-1)}}$).
\end{proposition}

\medskip

\begin{proof} As above it is convenient to put $v_s=1$ in the formal group law
$$
F(x,y) = F(x+y, v_s W^{(p)}(x,y)^{p^{s-1}}, v_s^{e_{2}}
W^{(p^2)}(x,y)^{p^{2 (s-1)}},...),
$$
where $W^{(p^i)}$ is the homogeneous polynomial of degree $p^i$ defined above and $e_{i}= (p^{is}-1)/ (p^{s}-1)$.
In particular $W^{(1)}=x+y$,
\begin{displaymath}
W^{(p)}=-\sum_{0<j<p}p^{-1}\binom{p}{j}x^{j}y^{p-j} ,
\end{displaymath}
\noindent and $W^{{p^i}} \notin (x^{p},y^{p})$.

Then for $s>1$ we can reduce modulo the ideal $(x^{p^{2(s-1)}}, y^{p^{2 (s-1)}})$ and get
\begin{eqnarray*}
F(x,y) &\equiv& F ( x + y, W^{(p)}(x,y)^{p^{s-1}})\\
&       =&
F (x + y+ W^{(p)}(x,y)^{p^{s-1}},
   W^{(p)}(x + y, W^{(p)}(x,y)^{p^{s-1}})^{p^{s-1}}, \dots )\\
& \equiv & F (x + y+ W^{(p)}(x, y)^{p^{s-1}},
        W^{(p)}(x^{p^{s-1}} +
y^{p^{s-1}}, W^{(p)}(x,y)^{p^{2(s-1)}}) ),
\end{eqnarray*}

\noindent and modulo $(x^{p^{2 (s-1)}}, y^{p^{2(s-1)}})$ we have
$$
F(x,y) \equiv x + y + W^{(p)}(x,y)^{p^{s-1}}.
$$
\end{proof}

\bigskip

One has also the following

\begin{proposition}(see \cite{BV})
Let $p=2$. Then
$$
F(x,y)=x+y+(xy+(x+y)(xy)^{2^{n-1}})^{2^{n-1}}\,\,\,\,\,\,modulo\,\,\,((x+y)xy)^{2^{2n-2}}.
$$
\end{proposition}

\bigskip

\section{The Abel formal group law and its p-typization.}

The Abel formal group law \cite{BU-KH} $\mathcal{F}_{Ab}$ is defined as the universal formal group that can be written in the form

\begin{equation}
\label{eq:F_Ab}
\mathcal{F}_{Ab}=xR(y)+yR(x),\text{ where } R(x)=1+\frac{a_1}{2}x+a_2x^2+a_3x^3+\cdots
\end{equation}

\medskip

Note that associativity requirement on $\mathcal{F}_{Ab}$ imposes the equation

$$ x(R(yR(z)+R(y)z)-R(y)R(z))=(R(xR(y)+R(x)y)-R(x)R(y))z. $$

To satisfy this equation one must uniquely determine all the $a_i$ through $a_1$ and $a_2$, which might
be arbitrary. For example, one has

\medskip

\begin{align*}
a_3&=-\frac{2}{3}a_1a_2,\\
a_4&=\,\,\,\,\frac{1}{2}a_1^2a_2-\frac{1}{2}a_2^2,\\
a_5&=-\frac{2}{5}a_1^3a_2+\frac{16}{15}a_1a_2^2,\\
a_6&=\,\,\,\,\frac{1}{3}a_1^4a_2-\frac{29}{18}a_1^2a_2^2+\frac{1}{2}a_2^3,\\
a_7&=-\frac{2}{7}a_1^5a_2+\frac{74}{35}a_1^3a_2^2-\frac{64}{35}a_1a_2^3,\\
a_8&=\,\,\,\,\frac{1}{4}a_1^6a_2-\frac{103}{40}a_1^4a_2^2+\frac{751}{180}a_1^2a_2^3-\frac{5}{8}a_2^4,\\
a_9&=-\frac{2}{9}a_1^7a_2+\frac{944}{315}a_1^5a_2^2-\frac{21632}{2835}a_1^3a_2^3+\frac{1024}{315}a_1a_2^4,\\
\,\,\,\,\,\,\,\,\,\,\,\,etc.
\end{align*}

\medskip

 The general formula for $a_n$, i.e., the additiuon law for $\mathcal{F}_{Ab}(x,y)$, has been obtained by V. M. Buchsteber in \cite{BUCH2}.

\begin{theorem}
\label{a_n}(see \cite{BUCH2}, Theorem 35.) Let $a=a_1$, $b=-2a_2$. Over the ring $\mathbb{Q}[a,b]$, the formal group law $\mathcal{F}_{Ab}(x,y)$ is expressed as
$$
\mathcal{F}_{Ab}(x,y)=x+y+b\biggl(\sum_{n=2}^{\infty}A_n(x^ny+xy^n)\biggr),
$$
where $A_2=-\frac{1}{2}$, $A_3=\frac{1}{3}a$,
$$
A_n=\frac{\delta_n}{n!}\prod_{j=2}^{[n/2]}[(j-1)(n-j)a^2+(n-2j+1)^2b], \,\,\,\,\,\,\,\,\, n\geq 4,
$$
and
$$
\delta_n=\begin{cases}
-(2s-1)& \text{if $n=2s,$}\\
2(s+1)sa & \text{if $n=2s+1$.}
\end{cases}
$$
\end{theorem}

 \medskip

See also \cite{BU-KH} for earlier description of the addition law for $\mathcal{F}_{Ab}(x,y)$.

\bigskip

The logarithm of this formal group law is given by

$$log_{Ab}(x)=x+\sum_{i\geq 1} m_ix^{i+1}=\int_{0}^{x}\frac{dt}{1+a_1t+a_2t^2+a_3t^3+\cdots}.$$

Thus for example

\begin{align*}
m_1&= -\frac{1}{2}a_1, \\
m_2&=\,\,\,\,  \frac{1}{3}\,\,\,\,(a_1^2-a_2),& \\
m_3&= -\frac{1}{4}a_1(a_1^2-\frac{8}{3}a_2), \\
m_4&=\,\,\,\,\frac{1}{5}\,\,\,\,(a_1^2-\frac{1}{3}a_2)(a_1^2-\frac{9}{2}a_2), \\
m_5&=-\frac{1}{6}a_1(a_1^2\,\,\,\,-a_2)(a_1^2-\frac{32}{5}a_2), \\
m_6&=\,\,\,\,\frac{1}{7}\,\,\,\,(a_1^2-\frac{1}{6}a_2)(a_1^2-\,\,\,\,\frac{9}{5}a_2)(a_1^2-\frac{25}{3}a_2),\\
m_7&=-\frac{1}{8}a_1(a_1^2-\frac{8}{15}a_2)(a_1^2-\frac{8}{3}a_2)(a_1^2-\frac{72}{7}a_2), \\
m_8&=\,\,\,\,\frac{1}{9}\,\,\,\,(a_1^2-\frac{1}{10}a_2)(a_1^2\,\,\,\,-a_2)(a_1^2-\,\,\,\,\frac{25}{7}a_2)(a_1^2-\frac{49}{4}a_2),\\
m_9&=-\frac{1}{10}a_1(a_1^2-\frac{1}{3}a_2)(a_1^2-\frac{32}{21}a_2)(a_1^2-\,\,\,\,\frac{9}{2}a_2)(a_1^2-\frac{128}{9}a_2), \\
\,\,\,\,\,\,\,\,\,\,\,\,etc.
\end{align*}

An evidence for the following general formula is that all the coefficients occurring at $a_2$ are of the form $\frac{2(i\pm j)^2}{ij}$.

\begin{proposition}
\label{jibformula}
Let $log_{Ab}(x)=x+\sum_{i\geq 1} m_ix^{i+1}$, $m_i \in \mathbb{Q}[a_1,a_2]$ be the logarithmic series of the universal Abel formal group law. Then one has
$$
m_{n-1}=\frac{1}{n}\prod_{j=1}^{n-1}(\frac{n-2i}{\sqrt{j(n-j)}}\sqrt{2a_2}-a_1);
$$
\end{proposition}

An equivalent formula is

$$
m_n=\frac{1}{n+1}\prod_{r=\frac{1-n}{n+1},\frac{3-n}{n+1},\cdots,\frac{n-3}{n+1},\frac{n-1}{n+1}}\left(\frac{2r}{\sqrt{1-r^2}}\sqrt{2a_2}-a_1\right).
$$

\bigskip

The corresponding two-parameter genus
$MU_*\rightarrow \mathbb{Q}[a,b]$
generalizes classical Todd genus of the multiplicative formal group law $\mathcal{F}_m(x,y)=x+y+txy$ over coefficient ring $Z[t,t^{-1}].$  The exponential of the Abel formal group law is
$$
exp_{Ab}(t)=\frac{e^{\alpha t}(e^{\sqrt{\beta}t}-1)}{\sqrt{\beta}}=\frac{e^{at}-e^{bt}}{a-b},
$$

where $\alpha=a_1/2$, $\beta=2a_2+1/4a_1^2$, $a=\alpha+\sqrt{\beta}$, $b=\alpha-\sqrt{\beta}.$

\medskip

This formal group law is named by  V. M. Buchstaber because of Abel's study of a functional equation that this exponential
satisfies.  The coefficient ring $\Lambda_{Ab}$ of $\mathcal{F}_{Ab}$ and its localizations at primes have been computed In \cite{BU-KH}.

\bigskip

The following "numerical" characterization of $\Lambda_{Ab}$ is given in \cite{CL-J0}.

\begin{proposition}
$\Lambda_{Ab}$ consists of those symmetric polynomials in $\mathbb{Q}[a,b]$ such that $f(kt,lt)\in \mathbb{Z}[t,(k-l)^{-1}]$ for any
integers $k,l$ such that $k\neq l.$
\end{proposition}

\bigskip

 It is pointed out in \cite{K-N}, \cite{roindiss} that the Sheffer sequence associated with $exp_{Ab}(t)$ gives the Gould polynomials $G_k(x,u,v)$ for $u=\alpha$,  $v=\sqrt{\beta}$.

\begin{proposition}
The logarithm $log_{Ab}(t)$ of the Abel universal formal group law, i.e. the inverse to the exponent $exp_{Ab}(t)=\frac{e^{ut}(e^{vt}-1)}{v}$,
is the generating function for $G_k(x,u,v)$
$$
log_{Ab}(t)=\sum_{k\geq 1}\frac{\partial{G_k(x,u,v)}}{\partial{x}}(0,u,v)v^k t^k
$$
and can be given by applying $\frac{\partial{e^{xlog_{Ab}(t)}}}{\partial{x}}$ at $x=0$ to the left and right sides of the equation
$$
\sum_{k\geq 0}\frac{G_k(x,u,v)}{k!}v^k t^k=e^{xlog_{Ab}(t)}.
$$
\end{proposition}

\medskip

Now one can easily deduce from \cite{Ro}

\begin{proposition}
\label{eq:log-Abel} The logarithmic series $log_{Ab}(t)$ of the universal Abel formal group law is given by
$$log_{Ab}(t)=\sum_{k\geq 1}1/v \binom{-(v+uk)/v}{k-1}\frac{v^k}{k}t^k,$$ where $u=\alpha$, and $v=\sqrt{\beta }$.
\end{proposition}

\medskip

Some few terms of $log_{Ab}(t)$ are
$$
log_{Ab}(t)=t-\frac{2u+v}{2}t^2+\frac{(3u+v)(3u+2v)}{3!}t^3-\frac{(4u+v)(4u+2v)(4u+3v)}{4!}t^4+\cdots
$$

The cohomological realizability of $\mathcal{F}_{Ab}$ and $\mathcal{F}_{2}$ localized away from 2 was sketched by Nadiradze, (see \cite{roindiss}, p. 41). A detailed proof of realizability of $\mathcal{F}_{Ab}$ is given by Busato in \cite{busato} using entirely different considerations.

\bigskip

\bigskip

We now turn to $p$-typization of the Abel formal group law. This is the p-typical formal group law with the logarithm
$$
t+m_{p-1}t^p+m_{p^2-1}t^{p^2}+m_{p^3-1}t^{p^3}+\cdots.
$$
So it is straightforward to calculate its coefficients and images of the Hazewinkel generators $v_n$ \cite{HAZ}
under the homomorphism from $BP_*=\mathbb{Z}_{(p)}[v_1,v_2\cdots]$, $|v_n|=2(p^n-1)$ to $\mathbb{Z}_p[a_1,a_2]$ which classifies this $p$-typical formal group law. For example, one has for $p=2$:
\begin{align*}
&v_1\mapsto -a_1,\\
&v_2\mapsto \frac{4}{3}a_1a_2,\\
&v_3\mapsto \frac{284}{105}a_1^5a_2-\frac{808}{105}a_1^3a_2^2+\frac{128}{35}a_1a_2^3,\\
&v_4\mapsto \frac{184108}{45045}a_1^{13}a_2-\frac{1108792}{15015}a_1^{11}a_2^2+\frac{4521344416}{10135125}a_1^9a_2^3-\frac{1265861152}{1126125}a_1^7a_2^4
+\frac{107918689792}{91216125}a_1^5a_2^5,\\
&etc.
\end{align*}

Using these expressions, we can then compute generators of the ideal in $BP_*$ which is the kernel
of the classifying map, thus presenting the coefficient ring of our two-parameter formal group by
generators and relations.

The explicit computation of these relations for $p=2$ in low degrees motivates the following

\begin{conjecture}
\label{qring}
The coefficient ring of the $2$-typization of the universal Abel formal group law is isomorphic to the quotient
$$\Lambda =\mathbb{Z}_{(2)}[v_1,v_2\cdots]/R,$$
with $|v_i|=2(2^n-1)$ and all generating relations are given by
$
v_1v_i^2v_j^2\sim P, \,\,\,1\leq i<j,
$
where $P$ consists of monomials not divisible by any $v_1v_{i'}^2v_{j'}^2 $, with $1\leq i'<j'$.
\end{conjecture}

One can compute the generating function

$$
\sum_{n=0}^{\infty}rank_{\mathbb{Z}_{(2)}}(\Lambda_n)t^n.
$$

Indeed using the above
relations we can eliminate any monomials divisible by some $v_1v_i^2v_j^2$. Remaining monomials will
then form a basis of $\Lambda$ as a $\mathbb{Z}_{(2)}$ module. These monomials can be subdivided into three disjoint
sets as follows:

\begin{align*}
(a)& \text{ any monomials not divisible by } v_1;\\
(b)& \text{ monomials of the form } v_1v_{i_1}v_{i_2}\cdots v_{i_k}v_j^n \text{ and } v_1^2v_{i_1}v_{i_2}\cdots v_{i_k}v_j^n,\text{ where }\\
   &k\geq 0,\, n>1,\, v_{i_1}v_{i_2}\cdots v_{i_k} \text{ are pairwise distinct and different from 1 and from j, and}\\
(c)& \text{ monomials } v_1^nv_{i_1}v_{i_2}\cdots v_{i_k}, \text{ where } k\geq 0,\, n>0,\, v_{i_1}v_{i_2}\cdots v_{i_k}
    \text{ are pairwise distinct and } \\
   &\text{not equal to } 1.
\end{align*}

According to this, the generating function will consist of three summands, namely:

\begin{align*}
&(a)\,\,\,&\frac{1}{1-t^3}&\frac{1}{1-t^7}\frac{1}{1-t^{15}}\cdots \frac{1}{1-t^{2^n-1}}\cdots \\
&(b) \,\,\,&(t+t^2)&(\,\,\,\,\,\,\frac{t^6}{1-t^3}(1+t^7)(1+t^{15})(1+t^{31})\cdots  \\
&          &       &+(1+t^3)\frac{t^{14}}{1-t^7}(1+t^{15})(1+t^{31})\cdots  \\
&          &       &+(1+t^3)(1+t^7)\frac{t^{30}}{1-t^{15}}(1+t^{31})\cdots  \\
&          &       &+(1+t^3)(1+t^7)(1+t^{31})\frac{t^{62}}{1-t^{31}}\cdots \\
&          &       &\cdots \\
&          &       &+(1+t^3)(1+t^7)\cdots (1+t^{2^{n-1}-1})\frac{t^{2(2^n-1)}}{1-t^{2^{n-1}}}(1+t^{2^{n+1}-1}) \cdots \\
&          &       &+\cdots ),\\
& and &&\\
&(c)       &\frac{t}{1-t}&(1+t^3)(1+t^7)(1+t^{15})\cdots (1+t^{2^n-1})\cdots .
\end{align*}

If we replace here all $(1-t^{2^n-1})^{-1}$ by $\frac{1+t^{2^{n-1}}}{1-t^{2(2^n-1)}}$, in the sum of these three functions one will have
a common multiple $(1+t^3)(1+t^7)(1+t^{15})\cdots (1+t^{2^n-1})\cdots $, so that the generating function will be

\begin{align*}
(1+t^3)(1+t^7)(1+t^{15})\cdots &(1+t^{2^n-1})\cdots \Bigl( \frac{1}{1-t^6}\frac{1}{1-t^{14}}
\frac{1}{1-t^{30}}\cdots \frac{1}{1-t^{2(2^n-1)}} \cdots \\
&+(t+t^2)\Bigl( \frac{1}{1-t^2}+\frac{t^6}{1-t^6}+\frac{t^{14}}{1-t^{14}}\cdots +\frac{t^{2(2^n-1)}}{1-t^{2(2^n-1)}}+ \cdots \Bigr) \Bigr).
\end{align*}

\bigskip

\begin{proposition}
The coefficient ring  of the 2-typization of the universal Abel formal group law
 cannot be realized as the coefficient ring of a cohomology theory as a Brown-Peterson cohomology with singularities.
\end{proposition}

\begin{proof}

Reducing modulo $2$ the ring $\Lambda=\mathbb{Z}_{(2)}[v_1,v_2\cdots]/R$ we have

\begin{equation*}
\Lambda /2\Lambda \cong \mathbb{F}_2 [v_1,v_2,v_3,\cdots ]/(v_1^3v_2^2, v_1^3v_3^2+v_1^2v_2^5, v_1v_2^2v_3^2+v_1^2v_2^4v_3+v_2^7, v_1^3v_4^2+v_1^2v_2v_3^4,\cdots).
\end{equation*}

 We see that $v_2^7\equiv 0$ modulo $v_1$, i.e., the last sequence of generating relations is not regular in $BP_*$.

\end{proof}

\end{document}